\documentclass[a4paper,11pt,makeidx]{amsart}
\usepackage[usenames,dvipsnames]{color}

\usepackage[backref=page]{hyperref}

\hypersetup{
 colorlinks,
 citecolor=Green,
 linkcolor=blue,
 urlcolor=Blue}

\usepackage{enumerate}
\usepackage{amscd}
\usepackage{xypic}  %commu
\usepackage{amssymb}
\usepackage{amsthm}
\usepackage{epsf}
\makeindex
\newtheoremstyle{break}% name
  {9pt}%      Space above, empty = `usual value'
  {9pt}%      Space below
  {\itshape}% Body font
  {}%         Indent amount (empty = no indent, \parindent = para indent)
  {\bfseries}% Thm head font
  {.}%        Punctuation after thm head
  {\newline}% Space after thm head: \newline = linebreak
  {}%         Thm head spec
\theoremstyle{break}
\newtheorem{bthm}{Theorem}

\newtheorem{bcor}{Corollary}
\theoremstyle{plain}
\newtheorem{thm}{Theorem}[section]
\newtheorem{cor}[thm]{Corollary}

\newtheorem{lemma}[thm]{Lemma}

\newtheorem{prop}[thm]{Proposition}

\newtheorem{crit}[thm]{Criterion}
\newtheorem{rem}[thm]{Remark}

\renewcommand{\proofname}{Proof}

\def\RR{{\mathbb R}}
\def\QQ{{\mathbb Q}}

\def\Proj{\operatorname{Proj}}

\def\Im{\operatorname{Im}}

\def\max{\operatorname{max}}

\def\Id{\operatorname{Id}}

\def\Nef{\operatorname{Nef}}
\def\Big{\operatorname{Big}}
\def\Eff{\operatorname{Eff}}
\def\Mor{\operatorname{Mor}}
\def\Supp{\operatorname{Supp}}
\def\Exc{\operatorname{Exc}}

\def\Sym{\operatorname{Sym}}
\def\End{\operatorname{End}}
\def\ord{\operatorname{ord}}

\def\NN{{\mathbb N}}
\def\QQ{{\mathbb Q}}
\def\PP{{\mathbb P}}
\def\B{\mathbf{B}}

\def\L{{\mathcal L}}

\def\O{{\mathcal O}}

\def\*{\otimes}
\def\eqv{\equiv}
\def\sub{\subseteq}

\def\ov{\overline}
\def\+{\oplus}                   % direct sum
\def\*{\otimes}                  % tensor product
\def\End{\operatorname{End}}

\def\Supp{\operatorname{Supp}}
\def\Int{\operatorname{Int}}

\def\Bs{\operatorname{Bs}}

\begin{document}

\title[Moriwaki divisors and the augmented base loci of divisors on $\overline{M}_g$]{Moriwaki divisors and the augmented base \\ loci of divisors on the moduli space of curves}

\author[S. Cacciola, A.F. Lopez and F. Viviani]{Salvatore Cacciola*, Angelo Felice Lopez* and Filippo Viviani**}

\thanks{* Research partially supported by the MIUR national project ``Geometria delle variet\`a algebriche" PRIN 2010-2011.}

\thanks{** Research partially supported by the MIUR national project ``Geometria delle variet\`a
algebriche" PRIN 2010-2011, by the MIUR national project ``Spazi di moduli e applicazioni" FIRB 2012, by the Research Network Program GDRE-GRIFGA
%,by FCT project ``Espa\c cos de Moduli em Geometria Alg\'ebrica"
and by CMUC - Centro de Matem\'atica da Universidade de Coimbra.}

\address{\hskip -.43cm Dipartimento di Matematica e Fisica, Universit\`a di Roma
Tre, Largo San Leonardo Murialdo 1, 00146, Roma, Italy. \newline e-mail {\tt cacciola@mat.uniroma3.it, lopez@mat.uniroma3.it, filippo.viviani@gmail.com}}


\begin{abstract}
We study the cone of Moriwaki divisors on $\overline{M}_g$ by means of augmented base loci. Using a result of Moriwaki, we prove that an $\RR$-divisor $D$ satisfies the strict Moriwaki inequalities if and only if $\B_+(D) \subseteq \partial \overline{M}_g $. Then we draw some interesting consequences on the Zariski decomposition of divisors on $\overline{M}_g$, on the minimal model program of $\overline{M}_g$ and on the log canonical models $\overline{M}_g(\alpha)$.
\end{abstract}

\maketitle

\section{Introduction}
\label{intro}

Let $g \geq 3$ and let $\overline{M}_g$ be the moduli space of stable curves on genus $g$. A striking result of Gibney, Keel and Morrison \cite[Thm. 0.9]{gkm} asserts that any nef divisor on $\overline{M}_g$, not linearly equivalent to zero, must be big. In terms of cones of divisors in the N\'eron-Severi space $N^1(\overline{M}_g)_{\RR}$, this implies that the nef cone does not meet the boundary of the big cone along rational nonzero classes. As a matter of fact, as we shall see, the same is true for real classes: $\Nef(\overline{M}_g) - \{0\} \subset \Big(\overline{M}_g)$. One way to see this is to consider the {\bf Moriwaki cone} $\Mor(\overline{M}_g)$,  that is the cone of $\RR$-divisors $D$ on $\overline{M}_g$ that are nef away from the boundary. The cone $\Mor(\overline{M}_g)$  was explicitly described by Moriwaki \cite[Cor. 4.3]{m1} in terms of the generators $\lambda, \delta_0, \ldots, \delta_{\lfloor g/2 \rfloor}$: an $\RR$-divisor $D \sim a \lambda - b_0 \delta_0 - \ldots - b_{\lfloor g/2\rfloor} \delta_{\lfloor g/2\rfloor}$ belongs to $\Mor(\overline{M}_g)$ if and only if it is an {\bf M-divisor}, that is it satisfies the {\bf Moriwaki inequalities}
\begin{equation}
\label{morineq}
a \geq 0, \ a \geq \frac{8g + 4}{g} b_0, \ a \geq \frac{2g + 1}{i(g-i)} b_i, \mbox{ for all} \ i = 1, \ldots , \lfloor g/2 \rfloor.
\end{equation}
The starting idea of this paper is that both the Moriwaki cone  and its interior, that is the cone of those $\RR$-divisors that satisfy the strict Moriwaki inequalities and which we call {\bf strict M-divisors}, can be interpreted in terms of restricted and augmented base loci.

Recall that the \emph{stable base locus} $\B(D)$ of an $\RR$-Cartier $\RR$-divisor $D$ on a normal projective variety $X$ is defined as
\[ \B(D)= \bigcap_{E \geq 0  :  E \sim_\RR D} \Supp(E), \]
with the convention that $\B(D)=X$ if the above intersection runs over the empty set.

%Let $X$ be a normal projective variety and let $D$ be an $\RR$-Cartier $\RR$-divisor on $X$. The {\bf stable base locus} $\B(D)$ of $D$ is $X$ if there is no $\RR$-Cartier $\RR$-divisor $E \geq 0$ such that $E \sim_\RR D$, or otherwise
%\[ \B(D)= \bigcap_{E \geq 0  :  E \sim_\RR D} \Supp(E). \]

The \emph{augmented base locus} and the \emph{restricted base locus} of $D$ are, respectively,
\[ \B_+(D)= \bigcap_{A {\ \rm ample}} \B(D-A) \ \text{ and } \ \B_-(D)= \bigcup_{A {\ \rm ample}} \B(D+A) \]
where $A$ runs among all ample $\RR$-Cartier $\RR$-divisors.
  We have the inclusions $\B_-(D) \subseteq \B(D) \subseteq \B_+(D)$ and $D$ is big if and only if $\B_+(D) \subsetneq X$.

Returning to $\overline{M}_g$, the main result of this article, where the assertion on $ \B_-(D)$ is just a rewriting of \cite[Thm. C]{m1}, is the following

\begin{bthm}
\label{main}
Let $g \geq 3$ and let $D$ be an $\RR$-divisor on $\overline{M}_g$. Then
\begin{enumerate}[(i)]
\item $\B_-(D) \subseteq \partial \overline{M}_g$  if and only if $D$ is an M-divisor;
\item $\B_+(D) \subseteq \partial \overline{M}_g$ if and only if $D$  is a strict M-divisor.
\end{enumerate}
\end{bthm}
Now nef non zero divisors are strict M-divisors (see  Lemma \ref{nefM}),  therefore the first simple consequence of Theorem \ref{main} is that
\[ \Nef(\overline{M}_g) - \{0\} \subset \Int(\Mor(\overline{M}_g)) \subset \Big(\overline{M}_g). \]
Note that this gives another proof on $\overline{M}_g$, but for $\RR$-divisors, of \cite[Thm. 0.9]{gkm}.

\vspace{0.2cm}

The following Figure is a schematic picture of the  Moriwaki cone and its relative position with respect to the nef cone $\Nef(\ov{M}_g)$ and the pseudoeffective cone $\ov{\Eff}(\ov{M}_g)$ of $\overline{M}_g$.

\begin{figure}[h!]
\unitlength 0.7mm % = 2.845pt
\linethickness{0.4pt}
\ifx\plotpoint\undefined\newsavebox{\plotpoint}\fi % GNUPLOT compatibility
\begin{picture}(334,119.25)(50,90)
\put(112,204){\line(1,0){99.75}}
\put(141,209.25){\makebox(0,0)[rc]{}}
\multiput(112,203.75)(-.03373430962,-.06066945607){956}{\line(0,-1){.06066945607}}
\multiput(79.75,145.75)(.03372333104,-.03406744666){1453}{\line(0,-1){.03406744666}}
\multiput(211.5,203.75)(.03372811535,-.04454170958){971}{\line(0,-1){.04454170958}}
\multiput(244,160.75)(-.03373015873,-.05133928571){1008}{\line(0,-1){.05133928571}}
\qbezier(210.25,109.5)(177,79.625)(128.75,96.25)
\multiput(112.25,203.75)(.033731853117,-.039175918019){2342}{\line(0,-1){.039175918019}}
\multiput(191.25,112)(.0336938436,.153078203){601}{\line(0,1){.153078203}}
\thicklines
\put(162,182){\line(0,-1){82.25}}
\thinlines
\multiput(162,196)(-.0337209302,-.0447674419){430}{\line(0,-1){.0447674419}}
\put(216,166){\makebox(0,0)[cc]{$\Mor(\ov{M}_g)$}}
\put(240,131){\makebox(0,0)[cc]{$\ov{\Eff}(\ov{M}_g)$}}
\put(161.93,181.93){\line(0,1){.9565}}
\put(161.93,183.843){\line(0,1){.9565}}
\put(161.93,185.756){\line(0,1){.9565}}
\put(161.93,187.669){\line(0,1){.9565}}
\put(161.93,189.582){\line(0,1){.9565}}
\put(161.93,191.495){\line(0,1){.9565}}
\put(161.93,193.408){\line(0,1){.9565}}
\put(161.93,195.321){\line(0,1){.9565}}
\put(161.93,197.234){\line(0,1){.9565}}
\put(161.93,199.147){\line(0,1){.9565}}
\put(161.93,201.06){\line(0,1){.9565}}
\put(161.93,202.973){\line(0,1){.9565}}
\put(159,198){\makebox(0,0)[cc]{$\lambda$}}
\put(162,207){\makebox(0,0)[cc]{$\delta$}}
\put(115,207){\makebox(0,0)[cc]{Effective boundary divisors}}
\multiput(148,177)(.04126547455,-.03370013755){727}{\line(1,0){.04126547455}}
\multiput(178,152.5)(.0337370242,.0523356401){578}{\line(0,1){.0523356401}}
\multiput(197.5,182.5)(-.0885286783,.0336658354){401}{\line(-1,0){.0885286783}}
\put(187,193){\makebox(0,0)[cc]{$\Nef(\ov{M}_g)$}}
\put(173,183){\makebox(0,0)[cc]{$13\lambda-\delta$}}
\put(173,167){\makebox(0,0)[cc]{$11\lambda-\delta$}}
\put(194,109){\makebox(0,0)[cc]{M}}
\put(176,146){\makebox(0,0)[cc]{$\frac{8g+4}{g}\lambda-\delta$}}
\put(173,101){\makebox(0,0)[cc]{$\frac{13}{2}\lambda-\delta$}}
\multiput(308.25,90.75)(.2536363636,-.0336363636){275}{\line(1,0){.2536363636}}
\multiput(378,81.5)(.033707865,-.049157303){178}{\line(0,-1){.049157303}}
\put(162,100){\line(-1,0){.25}}
\put(161.93,99.93){\line(0,-1){.975}}
\put(161.93,97.98){\line(0,-1){.975}}
\put(161.93,96.03){\line(0,-1){.975}}
\put(161.93,94.08){\line(0,-1){.975}}
\put(161.93,92.13){\line(0,-1){.975}}
\put(161.93,91.00){\line(0,-1){.975}}
\put(165,86){\makebox(0,0)[cc]{$s_g\lambda-\delta$}}
\end{picture}
\vspace{0.05cm}
\caption{A section of the three cones $\Nef(\ov{M}_g)\sub \Mor(\ov{M}_g)\sub \ov{\Eff}(\ov{M}_g)$ and their intersection with the plane $\langle \lambda, \delta \rangle$.
Here $s_g$ is the slope of $\ov{M}_g$ (see \cite{hmo}) which, for the sake of the picture, is assumed to be $\leq \frac{13}{2}$ (this is  known to be true for $g \geq 22$).}
%by \cite[Thm. 1 and 2]{eh}, \cite[Thm. 1]{f1} and \cite[Thm. 1.4]{f2}).}
 \label{fig}
\end{figure}
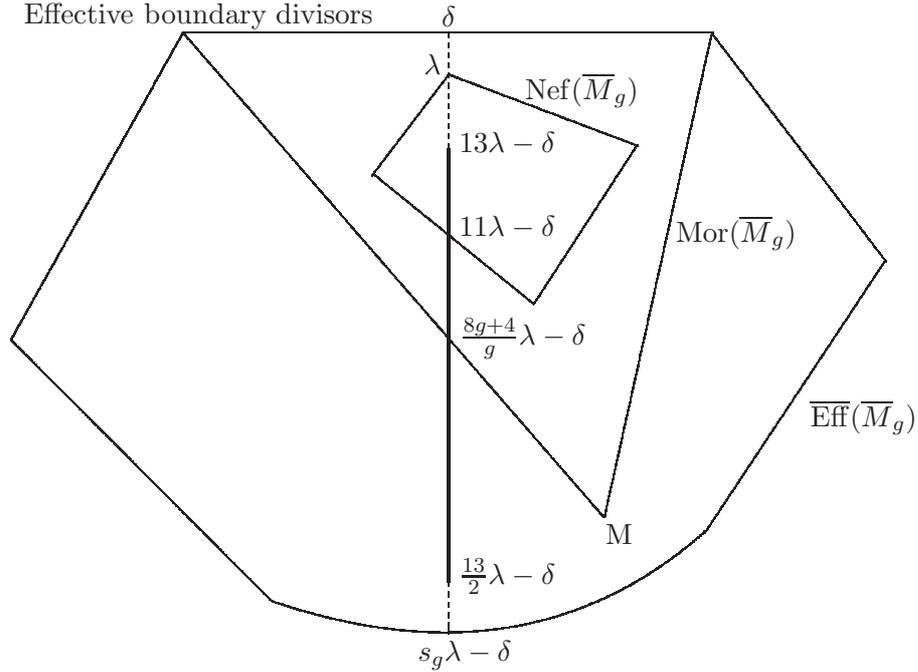

%\begin{rem}
%\label{rmkmor}
We point out that from \eqref{morineq} it follows that  $\Mor(\ov{M}_g)$ is a simplicial polyhedral cone whose extremal rays are generated by the boundary divisors $\{\delta_0, \ldots, \delta_{\lfloor g/2 \rfloor}\}$ and by the
{\bf Moriwaki divisor}
\begin{equation}\label{mordiv}
 M =  (8g + 4) \lambda - g \delta_0 - \sum\limits_{i = 1}^{\lfloor g/2 \rfloor} 4i(g-i) \delta_i .
 \end{equation}
Although we do not use the following facts, $M$ is known to be a big divisor
(see \cite{lop}))  and consequently the boundary of 
%Indeed $M$ is a big divisor (see \cite{lop} - we do not need this fact in the article). This implies that the boundary of 
$\Mor(\ov{M}_g)$ intersects the boundary of the pseudoeffective cone $\ov{\Eff}(\ov{M}_g)$ only in the common codimension-one face formed by effective boundary divisors.
The Moriwaki divisor also appears in the works of Hain-Reed \cite{hr} and of Hain \cite{hai}.
%\end{rem}

\vspace{0.1cm}

Another consequence of Theorem \ref{main} is that it gives many compactifications of $M_g$, generalizing \cite[Cor. 0.11]{gkm}.

\begin{bcor}
\label{mod1}
Let $g \geq 3$, let $D$ be a $\QQ$-divisor on $\overline{M}_g$ such that $\kappa(\overline{M}_g, D) \geq 0$ and for $m \in \NN$ consider the map $\varphi_{mD} : \overline{M}_g \dashrightarrow \PP H^0(\overline{M}_g, m D)$.
If $D$ is a strict M-divisor then there exists $m_0 \in \NN$ such that $\varphi_{mm_0D}$ is an isomorphism over $M_g$ for any $m \in \NN$. Vice versa  if there exists $m_1 \in \NN$ such that $\varphi_{mm_1D}$ is an isomorphism over $M_g$ for any $m \in \NN$, then $\B_+(D) \subseteq \partial \overline{M}_g \cup \B(D)$ and $D$ is a strict M-divisor when $\B(D) \subseteq \partial \overline{M}_g$.
\end{bcor}
It would be interesting to know whether some of the compactifications obtained in Corollary \ref{mod1}  arise from (stable) modular compactifications in the sense of \cite[Def. 1.1, 1.2]{s} or if, conversely, all the (stable) modular compactifications of \cite{s} arise from strict M-divisors.

In another direction, it would be desirable to extend Theorem \ref{main}  and Corollary \ref{mod1}  to
$\overline{M}_{g,n}$. A first partial result in this direction has been established in \cite[Cor. 2]{cl}, where it is proved, as a consequence of \cite[Thm. 0.9]{gkm}, that big and nef divisors on
$\overline{M}_{g,n}$ have their augmented base loci contained in the boundary of $\overline{M}_{g,n}$.

We can also apply Theorem \ref{main} to get some information on the log canonical models introduced by Hassett and Hyeon \cite{hh1}, \cite{hh2},
\[ f_{\alpha} : \overline{M}_g \dashrightarrow \overline{M}_g(\alpha) = \Proj\left(\bigoplus\limits_{m \geq 0}H^0(\ov M_ g, \lfloor m(13 \lambda - (2-\alpha) \delta) \rfloor )\right) \]
for $\alpha \in [0, 1] \cap \QQ$, where $f_{\alpha}$ is the standard rational map associated to the construction of $\Proj$ (see \cite[\href{http://stacks.math.columbia.edu/tag/01NK}{Tag 01NK}]{st}), or equivalently, the map associated to the linear system $|m(13 \lambda - (2-\alpha) \delta)|$, for $m \gg 0$ (see Cor. \ref{ka}).

In Figure \ref{fig}, we have depicted the intersection of the segment $\displaystyle \left\{\frac{13}{2-\alpha} \lambda - \delta\: : \: \alpha\in [0,1]\right\}$ with the cones $\Nef(\ov{M}_g)\sub \Mor(\ov{M}_g)\sub \ov{\Eff}(\ov{M}_g)$.

It has been asked by Hassett\footnote{in the open problem session of the AIM workshop ``The log minimal model program for the moduli space of curves", Palo Alto (California, USA), 10-14 December 2012. During the same problem session, M. Fedorchuk said that he could answer to the question away from the hyperelliptic locus.} whether the map $f_{\alpha}$ is an isomorphism over $M_g$ when $\alpha > \frac{3g + 8}{8g+4}$. We give an affirmative answer in the following

\begin{bcor}
\label{mod2}
Let $g \geq 3$. Then
\begin{enumerate}[(i)]
\item \label{mod2i} $f_{\alpha}$ is an isomorphism over $M_g$ if and only if $\alpha > \frac{3g + 8}{8g+4}$;
\item \label{mod2ii} If $\alpha = \frac{3g + 8}{8g+4}$ then $f_{\alpha}$ is defined over $M_g$ and it contracts the hyperelliptic locus $\ov{H}_g\subset \ov{M}_g$;
\item \label{mod2iii} If $\alpha<\frac{3g+8}{8g+4}$  then the hyperelliptic locus $\ov{H}_g$ is contained in $\B_-(13 \lambda - (2-\alpha) \delta)$.
\end{enumerate}
\end{bcor}
Note that part \eqref{mod2iii} implies that $f_{\alpha}$ is not defined over $\ov{H}_g$ whenever $\ov{H}_g$ is not contained in a divisorial component of $\B(13 \lambda - (2-\alpha) \delta)$ (which of course can occur only for $g\geq 4$).
We also remark that, whenever $13 \lambda - (2-\alpha) \delta$ is big, we have that $\B_-(13 \lambda - (2-\alpha) \delta) = \B(13 \lambda - (2-\alpha) \delta)$ (see Rmk. \ref{uni}).

\vspace{0.1cm}

Our next goal is deduce, from Theorem \ref{main}, some interesting consequences on the Zariski decomposition of divisors and on the minimal models of $\overline{M}_g$.

Recall that, given a pseudoeffective $\RR$-Cartier $\RR$-divisor $D$ on a normal projective variety $X$,
we say that $D$ has an \emph{$\RR$-CKM  Zariski decomposition} (CKM stands for Cutkoski-Kawamata-Moriwaki) if we can write\\
\vspace{0.1cm}
\centerline{$D = P + N$}
\vspace{0.1cm}
%\[ D = P + N \]
where $P, N$ are $\RR$-Cartier $\RR$-divisors such that $P$ is nef, $N$ is effective and $h^0(\lfloor mD \rfloor) = h^0(\lfloor mP \rfloor)$ for all $m \in \NN$, where $\lfloor mD \rfloor$ (and $\lfloor mP \rfloor$) is the round-down.

%We first recall the following \cite[Def. 1.1]{c}, \cite[\S 1]{k1}, \cite[Def. 1.1.5]{m2}
%\begin{defn}
%Let $X$ be a normal projective variety and let $D$ be a pseudoeffective $\RR$-Cartier $\RR$-divisor on $X$. We say that $D$ has an $\RR$-CKM Zariski decomposition if we can write
%\[ D = P + N \]
%where $P, N$ are $\RR$-Cartier $\RR$-divisors such that $P$ is nef, $N$ is effective and $h^0(\lfloor mD \rfloor) = h^0(\lfloor mP \rfloor)$ for all $m \in \NN$, where $\lfloor mD \rfloor$ and $\lfloor mP \rfloor$ are the round downs.
%\end{defn}

While on a smooth surface a Zariski decomposition always exists, by the celebrated result of Zariski, in general, on higher dimensional varieties, divisors may or may not have an $\RR$-CKM Zariski decomposition, even if we allow to pass to a birational model \cite[Thm. IV.2.10]{nak}. On the other hand, on a variety of nonnegative Kodaira dimension, the canonical bundle is expected to admit an $\RR$-CKM Zariski decomposition, after passing to a birational model, as a consequence of the conjectured existence of minimal models.

On $\overline{M}_g$ we obtain

\begin{bcor}
\label{mod3}
Let $g \geq 3$ and let  $D$ be an $\RR$-divisor on $\overline{M}_g$ such that $\kappa(D) \geq 1$. If $D$ has an $\RR$-CKM Zariski decomposition, then $D$ is a strict M-divisor. In particular, when $\kappa(\overline{M}_g) \geq 1$ (currently known for $g \geq 22$), the canonical divisor $K_{\overline{M}_g}$ does not have an $\RR$-CKM Zariski decomposition.
\end{bcor}
We stress that, for $g \geq 24$ or $g = 22$, since $K_{\overline{M}_g}$ is known to be big by \cite{hmu, eh, fa}, the minimal model of $\overline{M}_g$ exists by \cite[Lemma 10.1 and Thm. 1.2]{bchm}, whence the pull-back of $K_{\overline{M}_g}$ does have an $\RR$-CKM Zariski decomposition on some birational model of $\overline{M}_g$. On the other hand $\overline{M}_g$ is an interesting example of a normal projective variety whose canonical bundle does not have an $\RR$-CKM Zariski decomposition.

\begin{bcor}
\label{mod4}
Let $g$ be such that  $\kappa(\overline{M}_g) \geq 1$ (currently known for $g \geq 22$). Then there is no $K_{\overline{M}_g}$-non-positive projective birational morphism $f : \overline{M}_g \to X$ onto a normal $\QQ$-Gorenstein variety $X$ with $K_X$ nef. In particular, if $K_{\overline{M}_g}$ is big (currently known for $g \geq 24$ or $g = 22$), consider a rational map $f : \overline{M}_g \dashrightarrow (\overline{M}_g)_{min}$ to a minimal model obtained via contractions and flips of K-negative extremal rays. Then $f$ cannot be a morphism, that is, it is not possible to reach a minimal model of $\overline{M}_g$ only via contractions of extremal rays: at some step one must flip.
\end{bcor}

Note that, whenever $K_{\overline{M}_g}$ is big, we have that $\B_-(K_{\overline{M}_g}) = \B(K_{\overline{M}_g})$ (see Rmk. \ref{uni}).

\vspace{0.1cm}
Unless otherwise specified, we work throughout the paper over an algebraically closed field $k$ of characteristic $0$, although we expect that our results should hold  in  arbitrary characteristic (see \S\ref{char}).

\section{Generalities on Proj and Zariski decomposition}
\label{general}

We collect in this section some general facts that will be used in the proofs. They are all most likely well-known, but we include them for the lack of a reference (even though a similar version of Lemma \ref{proj} can be found in \cite[Lemma 1.6]{hk}).

Recall that, given a $\QQ$-Cartier $\QQ$-divisor $D$ on a normal projective variety $X$, its ring of sections is
\[ R(X, D) = \bigoplus\limits_{m \geq 0} H^0(X, \lfloor mD \rfloor) \]
 and if $mD$ is Cartier and $H^0(X, mD) \neq \{0\}$ then we denote by
\[ \varphi_{mD} : X \dashrightarrow Y_m \subseteq \PP H^0(X, mD) \]
 the map associated to $|mD|$, where $Y_m$ is the closure of its image (endowed with its reduced scheme structure).

%\begin{defn}
%Let $X$ be a normal projective variety, let $D$ be a $\QQ$-Cartier $\QQ$-divisor on $X$. We set
%\[ R(X, D) = \bigoplus\limits_{m \geq 0} H^0(X, \lfloor mD \rfloor) \]
%for the ring of sections and, if $mD$ is Cartier and $H^0(X, mD) \neq \{0\}$,
%\[ \varphi_m : X \dashrightarrow Y_m \subseteq \PP H^0(X, mD) \]
%for the map associated to $|mD|$, where $Y_m$ is the closure of its image (endowed with its reduced scheme structure).
%\end{defn}

\begin{lemma}
\label{proj}
Let $X$ be a normal projective variety defined over an algebraically closed field $k$ of arbitrary characteristic and let $D$ be a  $\QQ$-Cartier $\QQ$-divisor on $X$ such that $\kappa(X, D) \geq 0$ and $R(X, D)$ is a finitely generated $k$-algebra. Then there is $m_0 \in \NN$ such that $Y_{am_0} \cong \Proj(R(X, D))$ is normal for all $a \in \NN$. Moreover, with this identification, the standard rational map associated to the construction of $\Proj$ (see \cite[\href{http://stacks.math.columbia.edu/tag/01NK}{Tag 01NK}]{st}), $f_D : X \dashrightarrow \Proj(R(X, D))$ coincides with $\varphi_{am_0D} : X \dashrightarrow Y_{am_0}$ for all $a \in \NN$.
\end{lemma}
\begin{proof}
By \cite[Prop. 2.4.7(i)]{ega2} we can assume that $D$ is Cartier. By \cite[Lemma 2.1.6(v)]{ega2} there exists $s \in \NN$ such that
\begin{equation}
\label{gen}
S^h H^0(X, sD) \to H^0(X, hsD) \mbox{ is surjective for all } h \in \NN.
\end{equation}
Since $\kappa(X, D) \geq 0$, we get that $H^0(X, sD) \neq \{0\}$ and that
$\B(D) = \Bs(|hsD|)$ for all $h \in \NN$. Let $p : \widetilde{X} \to X$ be the normalized blow-up of $X$ along the base ideal of $|sD|$, so that we have a diagram
\[ \xymatrix{ \widetilde{X} \ar_{p}[d] \ar^{q}[rd]  \\ X \ar@{-->}_{\hskip -.2cm \varphi_{sD}}[r] & Y_s } \]
with $\widetilde{X}$ normal and $p$ birational. We can write
\begin{equation}
\label{dec}
p^{\ast}(sD) = M + F
\end{equation}
with $|M|$ base-point free, $F$ base component of $|p^{\ast}(sD)|$ and $p(\Supp(F)) = \B(D)$. Since $X$ is normal and $p$ is birational we have, by Zariski's Main Theorem,  that $p$ is an algebraic fiber space \cite[Def. 2.1.11]{l1} and therefore, for all $h \in \NN$,
\begin{equation}
\label{iso1}
H^0(X, hsD) \cong H^0(\widetilde{X}, p^{\ast}(hsD)).
\end{equation}
It follows by finite generation that, for all $h \in \NN$, $hF$ is the base component of $|p^{\ast}(hsD)|$,
whence
\begin{equation}
\label{iso2}
H^0(\widetilde{X}, p^{\ast}(hsD)) \cong H^0(\widetilde{X}, hM).
\end{equation}
But then $Y_{hs} = \Im \{\varphi_{hM} : \widetilde{X} \to \PP H^0(\widetilde{X}, hM) \}$ for all $h \in \NN$. On the other hand, by \cite[Thm. 2.1.27]{l1}, there is $h_0 \in \NN$ and an algebraic fiber space $\phi: \widetilde{X} \to Z$  such that  $\varphi_{hM} = \phi$ and $\Im \varphi_{hM} = Z$ for all $h \geq h_0$. Now $Z$ is normal by \cite[Thm. 2.1.15]{l1}, whence setting $m_0 = h_0s$ we get that $Y_{am_0} = Z$ is normal for all $a \in \NN$.

Let $A$ be an ample divisor on $Z$ such that $h_0M = \phi^{\ast}(A)$. As $\phi$ is an algebraic fiber space, we get
\begin{equation}
\label{iso3}
H^0(\widetilde{X}, sh_0M) = H^0(\widetilde{X}, \phi^{\ast}(sA)) \cong H^0(Z, sA).
\end{equation}
Since the product in a ring of sections is given by multiplication of sections, we deduce by \eqref{iso1}, \eqref{iso2} and \eqref{iso3}, that $R(X, m_0D) \cong R(\widetilde{X}, p^{\ast}(m_0D)) \cong R(\widetilde{X}, h_0M) \cong R(\widetilde{X}, \phi^{\ast}(A)) \cong R(Z, A)$. Finally by \cite[Prop. 2.4.7(i)]{ega2} we get
$\Proj(R(X, D)) \cong \Proj(R(X, m_0D)) \cong \Proj(R(Z, A)) \cong Z$ since $A$ is ample.

By \cite[\href{http://stacks.math.columbia.edu/tag/01NK}{Tag 01NK}]{st}, given a graded ring $S$, a scheme $T$ with a line bundle $\L$ and a homomorphism of graded rings $\psi: S \to R(T, \L)$, there is a morphism
\[ f_D : U(\psi) \to \Proj(R(X, D)) \]
where $U(\psi)$ is the union of the open subsets $T_{\psi(f)}$, with $f \in S_d, d > 0$. In our case, setting $T = X, \L = \O_X(D), S =  R(X, D)$ and $\psi = \Id_{R(X, D)}$, we have that $U(\psi) = X - \B(D)$ and we get  a rational map $f_D : X \dashrightarrow \Proj(R(X, D))$ defined on $X - \B(D)$. On the other hand, for any $d \in \NN$ such that $\B(D) = \Bs(|dD|)$, by \cite[\href{http://stacks.math.columbia.edu/tag/01NK}{Tag 01NK}]{st}, we have that $f_D$ coincides on $X - \B(D)$ with the morphism $X - \B(D) \to \Proj(R(X, D))$ defined on \cite[\href{http://stacks.math.columbia.edu/tag/01N8}{Tag01N8}]{st}, which, given the immersion $ \Proj(R(X, D)) \subset \PP^r$, $r = h^0(X, dD)$, is just the morphism $\varphi_{dD}$.
\end{proof}

We draw a consequence on the spaces $\ov{M}_g(\alpha)$.

\begin{cor}
\label{ka}
For every $\alpha \in [0, 1] \cap \QQ$ we have that $\overline{M}_g(\alpha)$ is normal and the rational map $f_{\alpha} : \overline{M}_g \dashrightarrow \overline{M}_g(\alpha)$ is given by $\varphi_{m(13 \lambda - (2-\alpha) \delta)}$ for $m$ sufficiently divisible.
\end{cor}
\begin{proof}
Set $K_{\alpha} = 13 \lambda - (2-\alpha) \delta$. We can assume that $\kappa(K_{\alpha}) \geq 0$. If $\alpha = 1$ the assertion follows by \cite[Cor. 5.18]{mum} (see also \cite[Thm. 1.3]{ch}), as $13 \lambda - \delta$ is ample. Now assume $\alpha < 1$ and set $B_{\alpha} = \alpha(\Delta_0+\Delta_2+\ldots+\Delta_{\lfloor g/2 \rfloor})+\frac{\alpha+1}{2}\Delta_1$, so that $K_{\alpha} = K_{\ov M_g}+B_{\alpha}$ and $(\ov{M}_g, B_{\alpha})$ is klt by \cite[Proof of Prop. A.13]{hh1} or \cite[Proof of Cor. 1.2.1]{bchm}. Then $R(\ov{M}_g, K_{\alpha})$ is a finitely generated $k$-algebra by \cite[Cor. 1.1.2]{bchm} and we just apply Lemma \ref{proj}.
\end{proof}

We also need a result about Zariski decompositions.

\begin{lemma}
\label{B+zd}
Let $X$ be a normal $\QQ$-factorial projective variety defined over an algebraically closed field $k$ of arbitrary characteristic, let $D$ be an $\RR$-divisor on $X$ having an $\RR$-CKM Zariski decomposition $D=P+N$.

Then $\B_+(D)=\B_+(P)$ and $\Supp(N)\subseteq \B_+(D)$.
\end{lemma}
\begin{proof}
We will use some results in \cite{nak}, \cite{bbp}, \cite{elmnp} and \cite{p}. We point out that, even though in the above references the results mentioned below are proved for smooth varieties over the complex numbers, the results hold with minor modifications on a normal $\QQ$-factorial projective variety defined over an algebraically closed field.

If $D$ is not big then $P$ is also not big, so that $\B_+(D)=\B_+(P)=X$.

Suppose now that $D$ is big, so that $P$ is also big by \cite[Thm. II.3.7 and Lemma II.3.16]{nak}. Given any prime divisor $\Gamma$ on $X$, one can define, as in \cite[Def. III.1.1]{nak},
\[ \sigma_{\Gamma}(D) = \inf\{\ord_{\Gamma}("E), E \ \mbox{effective} \ \RR\mbox{-divisor on X such that} \ E \equiv D \} \]
and, for any pseudoeffective $\RR$-divisor $F$ on $X$, as in \cite[Def. III.1.6]{nak},
\[  \sigma_{\Gamma}(F) = \lim_{\varepsilon \to 0^+} \sigma_{\Gamma}(F + \varepsilon A) \]
where  $A$  is an ample divisor (the definition does not depend on the choice of $A$). Now set
\[ N_\sigma(D) = \sum_{\Gamma} \sigma_{\Gamma}(D) \Gamma, \ P_\sigma(D) = D - N_\sigma(D). \]
Note that $N_\sigma(D)$ is an $\RR$-divisor by \cite[Cor. III.1.11]{nak}. The decomposition $D = P_\sigma(D) + N_\sigma(D)$ is called the $\sigma$-decomposition of $D$ (see \cite[Def. III.1.12]{nak}). By \cite[Rmk. III.1.17(3)]{nak} or \cite[Rmk.7.2 and Prop. 4.18]{p}, it follows  that $P = P_\sigma(D)$ and $N = N_\sigma(D)$. We also recall that $\Supp(N_\sigma(D)) \subseteq \B(D)$ (see \cite[Lemma 2.6]{bbp}).

Given any ample $\RR$-divisor $A$ on $X$ such that $A\leq D$, we find, by \cite[Lemma III.1.4]{nak}), that $\sigma_{\Gamma}(D) \leq \sigma_{\Gamma}(D-A) + \sigma_{\Gamma}(A) =  \sigma_{\Gamma}(D-A) \leq  \ord_{\Gamma}(D-A)$, whence $D-A\geq N$.
Therefore
$$\B_+(D)=\bigcap_{A\leq D} \Supp(D-A)= \bigcap_{A\leq D}(\Supp(D-A-N)\cup \Supp(N))=$$
$$=  \Supp(N) \cup \bigcap_{A\leq P} \Supp(P-A) =  \Supp(N) \cup \B_+(P).$$
Now let $\Gamma$ be a prime divisor in the support of $N$, so that $\sigma_{\Gamma}(D) > 0$. We will prove that $\Gamma \subseteq \B_+(P)$.
Let $H$ be an ample Cartier divisor such that $H-\Gamma$ is ample.
Then there exists $\varepsilon>0$ sufficiently small such that $\varepsilon \leq \sigma_\Gamma(D)$,  $\B_+(P)=\B(P-\varepsilon(H-\Gamma))$ by \cite[Prop. 1.5]{elmnp} (note that it is not needed that $P-\varepsilon(H-\Gamma)$ is a $\QQ$-divisor) and $P-\varepsilon H$ is big.
By \cite[Lemmas III.1.8 and III.1.4]{nak} we get
\[ 0 < \varepsilon = \sigma_{\Gamma}(P+\varepsilon \Gamma)\leq \sigma_{\Gamma}(P-\varepsilon(H-\Gamma))+ \sigma_{\Gamma}(\varepsilon H)=\sigma_{\Gamma}(P-\varepsilon(H-\Gamma)) \]
so that $\Gamma \subseteq \B(P-\varepsilon(H-\Gamma))=\B_+(P)$.
\end{proof}

\section{Proofs of the main results}
\label{moduli}

\renewcommand{\proofname}{Proof of Theorem {\rm \ref{main}}}
\begin{proof}
We begin by recalling some results of Moriwaki \cite{m1}. In \cite[Lemma 4.1]{m1}\footnote{which works over an algebraically closed field of characteristic different from two (due to the use of double covers).}, Moriwaki showed that there exist integral curves  $C, C_0, \ldots, C_{\lfloor g/2 \rfloor}$ inside $\ov{M}_g$, not entirely contained in the boundary $\partial \overline{M}_g$, with the following properties:
\begin{itemize}
\item $C$ is contained inside $M_g$;
\item $C_0$ is contained in $\ov H_g$ and intersects $\partial \ov M_g$ in points corresponding to isomorphism classes of irreducible curves with a single node;
\item For every $i=1,\ldots, \lfloor g/2 \rfloor$, $C_i$ is contained in $\ov H_g$ and intersects $\partial \ov M_g$ in points corresponding to isomorphism classes of stable curves formed by two irreducible components of genus $i$ and $g-i$ meeting in a single node.
\end{itemize}
It follows from the proof of \cite[Prop. 4.2]{m1}\footnote{which works over  an algebraically closed field of arbitrary characteristic using the  extension of the result of Cornalba-Harris \cite[Prop. 4.7]{ch} to arbitrary characteristic obtained by Yamaki in \cite[Thm. 1.7]{yam}}  that the cone spanned by $C, C_0,\ldots, C_{\lfloor g/2 \rfloor}$ inside $N_1(\ov M_g)_{\RR}$ is the dual of the cone of M-divisors.

Consider now an $\RR$-divisor $D$ on $\ov M_g$ such that $\B_-(D) \subseteq \partial \overline{M}_g$ (respectively $\B_+(D) \subseteq \partial \overline{M}_g$)
and let $\gamma$ be one of the curves $C, C_0, \ldots, C_{\lfloor g/2 \rfloor}$. Since $\gamma \not\subseteq \partial \overline{M}_g$, we get that $\gamma \not\subseteq \B_-(D)$
(respectively $\gamma \not\subseteq \B_+(D)$) and therefore $D \cdot \gamma \geq 0$
(respectively $D_{|\gamma}$ is big, that is $D \cdot \gamma > 0$). This shows that $D$ is an M-divisor (respectively a strict M-divisor).

Vice versa suppose first that $D$ is an M-divisor. As observed in the introduction, it follows from \eqref{morineq} that every M-divisor is an effective linear combination of the Moriwaki divisor $M$ and the boundary divisors.
Hence there exists $\beta \geq 0$ and an effective $\RR$-divisor $E$ on $\overline{M}_g$ such that $D = \beta M + E$ and $\Supp(E) \subseteq  \partial \overline{M}_g$,  where $M$ is the Moriwaki divisor as in \eqref{mordiv}.

We recall that the content of \cite[Thm. B]{m1} is exactly that $\B_-(M) \subseteq \partial \overline{M}_g$; hence
\[ \B_-(D) \subseteq \B_-(M) \cup \Supp(E) \subseteq \partial \overline{M}_g. \]
Moreover, if $D$ is a strict M-divisor, we can choose a sufficiently small ample $\RR$-divisor $A$ on $\overline{M}_g$ such that $D' := D - 2A$ is still a strict M-divisor and $\B_+(D) = \B(D - A)$ by \cite[Prop. 1.5]{elmnp} (note that it is not needed that $D-A$ is a $\QQ$-divisor). Then there exists $\beta' > 0$ and an effective $\RR$-divisor $E'$ on $\overline{M}_g$ such that $D' = \beta' M + E'$ and $\Supp(E') \subseteq  \partial \overline{M}_g$. Hence
\[ \B_-(D') \subseteq \B_-(M) \cup \Supp(E') \subseteq \partial \overline{M}_g \]
therefore also
\[ \B_+(D) = \B(D - A) = \B(D' + A) \subseteq \B_-(D') \subseteq \partial \overline{M}_g. \]
\end{proof}
\renewcommand{\proofname}{Proof}

We note that, for some divisors, we can compute exactly the augmented base locus.
\begin{prop}
\label{mod5}
Let $g \geq 3$ and let $D \sim a \lambda - b_0 \delta_0 - \ldots - b_{\lfloor g/2 \rfloor} \delta_{\lfloor g/2 \rfloor}$ be a big $\RR$-divisor on $\overline{M}_g$ with $b_i \leq 0$ for all $i = 0, \ldots , \lfloor g/2 \rfloor$. Then $\B_+(D) = \partial \overline{M}_g$. Moreover if $D$ is a $\QQ$-divisor then, for $m \gg 0$ sufficiently divisible, $\varphi_{mD}$ is the Torelli morphism to the Satake compactification $\overline{M}_g^S:=\Proj (R(\ov{M}_g, \lambda))$ of $M_g$, which is a normal variety.
\end{prop}
\begin{proof}
Recall that $\lambda$ is semiample, whence, by \cite[Thm. 2.1.15 and 2.1.27]{l1}, we get an algebraic fiber space $\pi = \varphi_{m \lambda} :  \overline{M}_g \to \Im \varphi_{m\lambda} \cong \ov{M}_g^S$ for $m \gg 0$ sufficiently divisible (this is the Torelli morphism to the Satake compactification)  and that $\ov{M}_g^S$ is normal. Moreover, it is well-known that $\Exc(\pi) = \partial \overline{M}_g$ (see e.g. \cite[Chap. XIV, \S 5]{acg}).

Notice that the restriction of $D$ to $M_g$ is linearly equivalent to $a\lambda$. Since $D$ is big and $\lambda$ is semiample, this implies that $a > 0$. Let $A$ be an ample $\QQ$-divisor such that $\lambda = \pi^{\ast} A$ and set $F = - b_0 \delta_0 - \ldots - b_{\lfloor g/2 \rfloor} \delta_{\lfloor g/2 \rfloor}$, so that $F$ is effective and $\pi$-exceptional. As $\overline{M}_g$ and $\ov{M}_g^S$ are normal and $\pi$ is birational, we can apply \cite[Prop. 2.3]{bbp}:
\[ \B_+(D) = \B_+(\pi^{\ast} (aA) + F) = \pi^{-1}(\B_+(aA)) \cup \Exc(\pi) = \Exc(\pi) =  \partial \overline{M}_g. \]
Now if $D$ is a $\QQ$-divisor and $m \gg 0$ is such that $mD$ and $m a \lambda$ are Cartier, then
\[ H^0(\overline{M}_g, m a \lambda) \cong H^0(\overline{M}_g, m (a \lambda+F)) \cong H^0(\overline{M}_g, m D) \]
and the last assertion of the Proposition follows.
\end{proof}

\renewcommand{\proofname}{Proof of Corollary  {\rm \ref{mod1}}}
\begin{proof}
By \cite[Thm. A]{bcl}, given a big $\QQ$-divisor $D$ on $\ov M_g$, we have that there exists $m_0 \in \NN$ such that $\overline{M}_g - \B_+(D)$ is the largest open subset of $\ov M_g - \B(D)$ where the maps $\varphi_{mm_0D}$ are isomorphisms for every $m \in \NN$. Using this, Corollary \ref{mod1} follows from Theorem \ref{main}.
\end{proof}

\renewcommand{\proofname}{Proof of Corollary {\rm \ref{mod2}}}
\begin{proof}
Note that $K_{\alpha} := 13 \lambda - (2-\alpha) \delta$ is a strict M-divisor if and only if $\alpha > \frac{3g + 8}{8g+4}$. Then (i) follows from Corollaries \ref{mod1} and \ref{ka}.

Assume now that $\alpha =\frac{3g + 8}{8g+4}$. Then, $K_{\alpha}$ is a (non-strict) M-divisor and, moreover, it is big, for its slope $s(K_{\alpha})=8+\frac{4}{g}$ is larger than the one of a Brill-Noether divisor  if $g+1$ is composite (see \cite[Thm. 1]{eh}) or of the Petri divisor if $g$ is even (see \cite[Thm. 2]{eh}). Also we claim that $\B(K_{\alpha}) = \B_-(K_{\alpha})$. Let $x \in  \B(K_{\alpha})$ and let $v$ be any divisorial valuation with center $\{x\}$. By the finite generation of $R(\ov{M}_g, K_{\alpha})$, as in \cite[Prop. 2.8]{elmnp} or \cite[\S 2.2]{bbp}, we have that $v(\|K_{\alpha}\|) > 0$, whence $x \in \B_-(K_{\alpha})$ and the claim is proved (see also Rmk. \ref{uni}). By Theorem \ref{main}, we get $\B(K_{\alpha})=\B_-(K_{\alpha}) \subseteq \partial \overline{M}_g$, whence that $f_{\alpha}$ is defined over $M_g$.

In order to prove the second statement of (ii), observe that  $K_{\alpha}$ is proportional to the Cornalba-Harris divisor  $(8g+4) \lambda - g \delta$ of  \cite[Prop. 4.3]{ch}. It follows from \cite[Prop. 4.3, Thm. 4.12]{ch}
(see also \cite[Cor. 1.8]{yam} in positive characteristics)  that $K_{\alpha}$ intersects to zero the curves constructed by Cornalba-Harris in \cite[p. 469]{ch}
\footnote{Indeed, it is easily checked, by \cite[Prop. 4.7]{ch} (see also \cite[Thm. 1.7]{yam} in positive characteristics), that these curves are all numerically proportional to the curve $C_0$ constructed in \cite[Lemma 4.1]{m1}.}: these are curves in $\ov H_g$ given by a family $\pi:X\to T$ of stable hyperelliptic curves over a smooth projective curve $T$  obtained as a double cover $\eta:X\to T\times \PP^1$ branched over a general curve $B\subset T\times \PP^1$ of class $(2g+2,2m)$ for some $m\geq 1$.  As the image of $T\to \ov H_g$ passes through the general point of $\ov H_g$, it follows that the map $f_{\alpha}$ contracts the hyperelliptic locus $\ov H_g\subset \ov M_g$.
This finishes the proof of (ii).

Assume finally that $\alpha <\frac{3g + 8}{8g+4}$. Then $K_{\alpha}$ intersects negatively the Cornalba-Harris curves considered above, which therefore must belong to $\B_-(K_{\alpha})$. By what we said above, we deduce that $\ov{H}_g\subset \B_-(K_{\alpha})$, which proves (iii).
\end{proof}

In order to prove Corollary \ref{mod3}, we need the following

\begin{lemma}\label{nefM}
Let $g\geq 3$ and let $D$ be a non-zero $\RR$-divisor on $\overline{M}_g$. If $D$ is nef then $D$ is a strict M-divisor.
\end{lemma}
\renewcommand{\proofname}{Proof}
\begin{proof}
If $D \sim a \lambda - b_0 \delta_0 - \ldots - b_{\lfloor g/2 \rfloor} \delta_{\lfloor g/2 \rfloor}$ is nef, then by intersecting $D$ with $F$-curves one finds that its coefficients must satisfy the following relations (and many others, see  \cite[Thm. 2.1]{gkm})
$$ a\geq 12b_0-b_1 \: \text{ and } \: 2b_0 \geq b_i\geq 0 \text{ for any } 1\leq i\leq g/2.$$
From these relations we get the chain of inequalities (for any $1 \leq i \leq g/2$):
$$a\geq 12b_0-b_1 \geq 10b_0\geq 5b_i. $$
Now we conclude that the strict Moriwaki inequalities \eqref{morineq} hold true for $D$ since (for any $1\leq i\leq g/2$ and $g\geq 3$) we have that
$$\begin{aligned}
& a\geq 10b_0\geq 0 \: \text{ with equalities if and only if } a=b_0=0,\\
& 10b_0 \geq \frac{8g+4}{g} b_0 \: \text{ with equality if and only if } b_0=0,\\
& 5b_i \geq \frac{2g+1}{i(g-i)} b_i \: \text{ with equality if and only if } b_i=0.
\end{aligned}$$

\end{proof}

\renewcommand{\proofname}{Proof of Corollary {\rm \ref{mod3}}}
\begin{proof}
Suppose that $D = P + N$ is an $\RR$-CKM Zariski decomposition. Then $P$ is nef and non-trivial, because $\kappa(P) = \kappa(D) \geq 1$, whence it is a strict M-divisor by Lemma \ref{nefM}. Therefore, by  Lemma \ref{B+zd} and Theorem \ref{main}, we have $\B_+(D) = \B_+(P) \subseteq \partial \overline{M}_g$, so that $D$ is a strict M-divisor again by Theorem \ref{main}. To conclude we just note that
\[ K_{\overline{M}_g} = 13 \lambda -  2 \delta_0 -  3 \delta_1 - 2 \delta_2 - \ldots - 2 \delta_{ \lfloor g/2 \rfloor} \]
is not an M-divisor.
\end{proof}
\renewcommand{\proofname}{Proof}

\renewcommand{\proofname}{Proof of Corollary {\rm \ref{mod4}}}
\begin{proof}
Let $a \in \NN$ be such that $aK_{\overline{M}_g}$ and $aK_X$ are Cartier. Now non-positivity of $f$ means that we have
\begin{equation}
\label{zar}
aK_{\overline{M}_g} = f^{\ast}(aK_X) + E
\end{equation}
with $E \geq 0$ and $f$-exceptional. Setting $P = f^{\ast}(aK_X)$ and $N = E$, we see immediately that \eqref{zar} is an $\RR$-CKM Zariski decomposition of $aK_{\overline{M}_g}$, thus contradicting Corollary \ref{mod3}.
To conclude the proof recall that, as discussed in the introduction (after Corollary \ref{mod3}), if $K_{\overline{M}_g}$ is big, then $\overline{M}_g$ has a minimal model $(\overline{M}_g)_{min}$. Hence $(\overline{M}_g)_{min}$ has normal $\QQ$-factorial dlt singularities, $K_{(\overline{M}_g)_{min}}$ is nef and there is a projective birational map $f : \overline{M}_g \dashrightarrow (\overline{M}_g)_{min}$ that is $K_{\overline{M}_g}$-negative (in fact $f$ is obtained via contractions and flips of $K_{\overline{M}_g}$-negative extremal rays). Then $f$ cannot be a morphism, by what we proved above.
\end{proof}
\renewcommand{\proofname}{Proof}

\begin{rem}
\label{uni}
It follows from \cite[Thm. A]{bbp} that, whenever $K_{\alpha} = 13 \lambda - (2-\alpha) \delta$ or $K_{\overline{M}_g}$ is big, then $\B_-(K_{\alpha}) = \B(K_{\alpha})$, $\B_-(K_{\overline{M}_g}) = \B(K_{\overline{M}_g})$ and every irreducible component of them and of  $\B_+(K_{\alpha})$, $\B_+(K_{\overline{M}_g})$ is uniruled.
\end{rem}

\subsection{Characteristic zero versus positive characteristic}
\label{char}

Even though  we worked, throughout the paper, over an algebraically closed  field of characteristic zero, we believe that all our results can be extended to a field of arbitrary characteristic. For the benefit of the reader, let us specify what is missing in positive characteristic.
\begin{enumerate}[(i)]
\item The proof of Theorem \ref{main} uses in a crucial way \cite[Thm. B]{m1}, which is currently known only in characteristic zero. The missing ingredient in positive characteristic is, given a smooth projective curve $C$ of genus $g\geq 3$, the validity of \cite[Claim 2.5.1]{m1} for the vector bundle  $M_{\omega_C}:=\ker(H^0(C,\omega_C)\otimes \O_C\stackrel{{\rm ev}}{\longrightarrow} \omega_C)$,  namely that there exists an ample line bundle $A$  such that
\begin{equation}
\label{vanH0}
H^0(C,\Sym^m(\End(M_{\omega_C}))\otimes A)=0 \text{ for every } m\gg 0.
 \end{equation}
The semistability of the vector bundles $\Sym^m(\End(\omega_C))$, which would imply the vanishing in \eqref{vanH0}, is not known in positive characteristic. Note that $M_{\omega_C}$ is semistable (of slope $-2$) in every characteristic (see \cite{pr} or also \cite[Prop. 3.2]{el} whose proof works verbatim for $M_{\omega_C}$), but this implies the semistability of 
$\Sym^m(\End(\omega_C))$ only in characteristic zero. On the other hand, the stronger condition of being strongly semistable (which is preserved by tensor products and symmetric products even in positive characteristic by \cite[Thm. 7.2, Cor. 7.3]{m1}) fails in positive characteristic for $M_{\omega_C}$ for some smooth plane quartics $C$, as it follows by combining \cite[Cor. 4.16]{t} and the several examples worked out in \cite{mo}.
\item Once the vanishing in \eqref{vanH0} has been established, our Theorem \ref{main} would follow in any characteristic different from two (in characteristic two one would also need to extend the construction of the curves $C_0,\ldots,C_{[g/2]}$ in \cite[Lemma 4.1]{m1}). From this, Corollaries \ref{mod3}, \ref{mod4} and Proposition \ref{mod5} would also follow in any characteristic different from two. On the other hand, in order to extend  Corollaries \ref{mod1}, \ref{mod2}, \ref{ka} and Remark \ref{uni} to positive characteristic, one would also need to establish the finite generation of the section ring of the divisor $13\lambda -(2-\alpha) \delta$ (for $\alpha \in [0,1]\cap \QQ$) on $\ov M_g$, which is currently known only in characteristic zero due to \cite{bchm}.
\end{enumerate}

\paragraph{\bf Acknowledgements}

The present collaboration started during the working group ``On the Birational geometry of the moduli space of curves" that was held during the spring semester of the academic year 2011/2012 at the University of Roma Tre. We thank all the participants to the working group for the stimulating atmosphere. We thank C. Camere for pointing out the paper \cite{t}. The third author heard about Hassett's question (see Corollary \ref{mod2}) at the AIM workshop ``The log minimal model program for the moduli space of curves", Palo Alto (California, USA), 10-14 December 2012.
He would like to thank the organizers  J. Alper, M. Fedorchuk, B. Hassett and D. Smyth for the invitation to participate to the workshop.

\section{Appendix: the bigness of Moriwaki's divisor}
\label{app}

Given the cumbersome calculations we give, in this appendix, the proof of the bigness of the Moriwaki divisor. We remark that this is just for completeness' sake, as we do not need this fact in the article.

We will use the following

\begin{crit}
\label{cr}
Let $g \geq 3$ and let $D \equiv a \lambda-\sum\limits_{i=0}^{\lfloor g/2 \rfloor} b_i  \delta_i$ be an $\RR$-divisor on $\ov{M}_g$ with $a > 0$. Assume that there exists an effective 
$\RR$-divisor $E \equiv \alpha \lambda -\sum\limits_{i=0}^{\lfloor g/2 \rfloor} \beta_i \delta_i$
such that
\begin{equation}
\label{ai}
\tag{$A$}
\alpha>0
\end{equation}
\begin{equation}
\label{bi}
\tag{$B_i$}
\beta_i > 0, \mbox{\ for all } 0 \leq i \leq \lfloor g/2 \rfloor
\end{equation}
and
\begin{equation}
\label{ci}
\tag{$C_i$}
\alpha b_i<a \beta_i, \mbox{\ for all } 0 \leq i \leq \lfloor g/2 \rfloor.
\end{equation}
Then $D$ is big.
\end{crit}
\begin{proof}
We can choose $v \in \RR$, $v \geq 0$ such that, for all $0 \leq i \leq \lfloor g/2 \rfloor$, we have
\[ \frac{b_i}{\beta_i} \leq v < \frac{a}{\alpha}. \]
Now $D \equiv (a-v\alpha) \lambda + vE + \sum\limits_{i=0}^{\lfloor g/2 \rfloor} (v\beta_i - b_i) \delta_i$ is big since $\lambda$ is big.
\end{proof}

\begin{lemma}
\label{mbig}
Let $g \geq 3$ and let $M =  (8g + 4) \lambda - g \delta_0 - \sum\limits_{i = 1}^{\lfloor g/2 \rfloor} 4i(g-i) \delta_i$ be the Moriwaki divisor on $\ov{M}_g$. Then $M$ is big.
\end{lemma}
\begin{proof}
We apply  Criterion \ref{cr}. As $a = 8g+4>0$, $b_0 = g> 0$ and, for all $1 \leq i \leq \lfloor g/2 \rfloor$, $b_i = 4i(g-i) > 0$, we will need to verify only $(A)$ and all $(C_i)$'s.

If $g+1$ is not prime, as in \cite[Theorem 1]{eh1}, we can write $g+1=(r+1)(s-1)$, for some integers $s\geq 3$ and $r\geq 1$ and we can consider the Brill-Noether divisor $D_s^r$ on $ \overline{M}_g$. By \cite[Theorem 1]{eh1} there exists $c>0$ such that
\[ 0 \leq \frac{1}{c} D_s^r \eqv (g + 3) \lambda - \frac{g+1}{6} \delta_0 - \sum\limits_{i = 1}^{\lfloor g/2 \rfloor} i(g-i) \delta_i. \]
Setting $E =    \frac{1}{c} D_s^r$, we have that $(A)$ is satisfied. Also $(C_0)$ is equivalent to $g^2-3g+2 > 0$, while, for $i \geq 1$, $(C_i)$  is equivalent to $g-2 > 0$, so all the $(C_i)$'s are also satisfied. 

Assume from now on that $g+1$ is prime, so that we can write $g=2(d-1)$, for some $d\geq 3$ and we can consider the Petri divisor $E_d^{1}$ on $ \overline{M}_g$. By \cite[Theorem 2]{eh1} there exists $c>0$ such that
\[ 0 \leq \frac{1}{c} E_d^1= (6d^2+d-6) \lambda - \sum\limits_{i = 0}^{\lfloor g/2 \rfloor} f_i \delta_i \]
where
\begin{eqnarray}
f_0=d(d-1);\\
\label{pr1}
f_1=(2d-3)(3d-2);\\
\label{pr2}
f_2=3(d-2)(4d-3).
\end{eqnarray}
Moreover, setting $k=d-1$ and 
\begin{equation}
\label{gamma_i}
\gamma_i= (i-1)(i-2)\frac{(2k-2)!}{k!(k-1)!}-\sum_{l=1}^{\lfloor \frac{i-2}{2} \rfloor} 2(i-1-2l)\frac{(2l)!(2k-2-2l)!}
{(l+1)!l!(k-l)!(k-l+1)!}
\end{equation}
by \cite[(5.3)]{eh1} we have 
\begin{equation}
\label{f_i}
f_i=-i(i-2)f_1+\frac{i(i-1)}{2} f_2+ \frac{ \gamma_i}{c} \ \ \mbox{\ for all } 3 \leq i \leq d-1.
\end{equation}
Setting $E =   \frac{1}{c} E_d^1$ and recalling that $d\geq 3$, we have that $(A)$ is satisfied.

Condition $(C_0)$ is $(6d^2+d-6)g<(8g+4)d(d-1)$, which is equivalent to $2d^2-7d+6 > 0$, whence it  is satisfied.

Condition $(C_1)$ is $(6d^2+d-6)4(g-1)<(8g+4)(2d-3)(3d-2)$, which is equivalent to $2d^3-9d^2+13d-6 > 0$, whence it  is satisfied.

Condition $(C_2)$ is $(6d^2+d-6)8(g-2)<(8g+4)3(d-2)(4d-3)$, which is equivalent to $24d^3-124d^2+203d-102 > 0$, whence it  is satisfied.

For all $i=3,\dots, d-1$, condition $(C_i)$ is equivalent to 
\[ f_i > \frac{(6d^2+d-6)4i(g-i)}{8g+4}=\frac{(6d^2+d-6)i(2d-2-i)}{4d-3} \]
and using \eqref{f_i} can be transformed in
\begin{equation}
\label{ecco}
-i(i-2)f_1+\frac{i(i-1)}{2}f_2+ \frac{\gamma_i}{c}>\frac{(6d^2+d-6)i(2d-2-i)}{4d-3}, \ 3 \leq i \leq d-1.
\end{equation}
To prove \eqref{ecco} we will show that $\gamma_i\geq 0$ for all $i=3,\dots, d-1$ and
\begin{equation}
\label{ecco1}
-i(i-2)f_1+\frac{i(i-1)}{2}f_2 >\frac{(6d^2+d-6)i(2d-2-i)}{4d-3}, \ 3 \leq i \leq k.
\end{equation}
Now \eqref{ecco1} is equivalent to
\[ i \left((4d-3)(f_2-2f_1) +2(6d^2+d-6)\right)>(4d-3)(f_2-4f_1) + 2(6d^2+d-6)(2d-2) \]
and using \eqref{pr1} and \eqref{pr2}, to
\[ 24d^3-92d^2+109d-42 > (16d^2-47d+30)i \]
so that, as $i \leq d-1$, we reduce it to $8d^3-29d^2+32d-12 > 0$, whence it  is satisfied.

It remains to prove that $\gamma_i\geq 0$ for all $i=3,\dots, k=d-1$.

Note that, by (\ref{gamma_i}), $\gamma_3=\frac{2(2k-2)!}{k!(k-1)!}>0$. Hence if we put $c_i=\frac{1}{2} (\gamma_i-\gamma_{i-1})$, we will be done if we show that $c_i\geq 0$ for all $i=4,\dots, k$. In particular we can suppose $k\geq 4$.

To simplify the notation let 
\[ b_l=\frac{(2l)!(2k-2-2l)!}{(l+1)!l!(k-l)!(k-l+1)!} \]
so that, by (\ref{gamma_i}), we can write
\[ c_i=(i-2)\frac{(2k-2)!}{k! (k-1)!}-\sum_{l=1}^{\lfloor \frac{i-2}{2}\rfloor} b_l. \]
As $c_4=\frac{(2k-4)!}{k!(k-1)!}(2(2k-2)(2k-3)-1)\geq 0$, setting $d_i=c_i-c_{i-1}$, we are reduced to prove that $d_i \geq 0$ for all $i=5,\dots, k$.

If $i$ is odd, then $d_i=\frac{(2k-2)!}{k! (k-1)!}\geq 0$, so that we can assume that $i$ is even.
In particular we can put $i=2h+2$, where $2\leq h\leq \lfloor \frac{k-2}{2} \rfloor$, and we get
\[ d_i=\frac{(2k-2)!}{k!(k-1)!}-\frac{(2h)!(2k-2-2h)!}{(h+1)!h!(k-h)!(k-h+1)!}. \]
In this way, after putting $v_h=\frac{(2h)!(2k-2-2h)!}{(h+1)!h!(k-h)!(k-h+1)!}$, we need to prove that
\begin{equation}
\label{v_h}\frac{(2k-2)!}{k!(k-1)!}\geq v_h
\end{equation}
 for all $h\in\{2,\dots,\lfloor \frac{k-2}{2}\rfloor\}$ and for all $k\geq 6$.

We now claim that, for $2 \leq h \leq \lfloor \frac{k-2}{2}\rfloor$, we have $v_h \leq \max\{v_2,v_{\lfloor \frac{k-2}{2}\rfloor}\}$.

In fact, for all $h\geq 3$, we can write $v_h-v_{h-1}=C_{k,h} N_{k,h}$, where
\[ C_{k,h}=\frac{2(2h-2)!(2k-2h-2)!}{h!(h-1)!(k-h+1)!(k-h)! (h+1)(k-h+2)(k-h+1)}\geq 0 \]
for all $3 \leq h \leq \lfloor \frac{k-2}{2}\rfloor, k \geq 6$,  and
\begin{eqnarray*}
N_{k,h}&=& (2h-1)(k-h+2)(k-h+1)
-(k-h)(2k-2h-1)(h+1)\\ &=&-3k^2+13kh-2k-10h^2+6h-2.
\end{eqnarray*}

In particular $v_h\leq v_{h-1}$ if and only if $N_{k,h}\leq 0$,
if and only if $h\leq k_1:=\frac{13k+6-\sqrt{49k^2+76k-44}}{20}$ or $h\geq k_2:=\frac{13k+6+\sqrt{49k^2+76k-44}}{20}$.
Thus the claim follows by noticing that, for all $k\geq 6$, we have  $k_2> \lfloor \frac{k-2}{2}\rfloor$.

Thanks to the claim it suffices to prove that (\ref{v_h}) holds for $h=2$ and $h=\lfloor \frac{k-2}{2}\rfloor$.
Since $v_2=\frac{2(2k-6)!}{(k-2)!(k-1)!}$, we have that (\ref{v_h}) holds for $h=2$.

Suppose $h=\lfloor \frac{k-2}{2}\rfloor$. If $k$ is even, then $k=2h+2$, so that (\ref{v_h}) is equivalent to 
\[ \frac{(4h+2)!}{(2h+2)!(2h+1)!}\geq \frac{(2h)!(2h+2)!}{(h+1)!h!(h+2)!(h+3)!} \]
which in turn is verified if and only if
\[ a_h:=\frac{(4h+2)!h!(h+1)!(h+2)!(h+3)!}{((2h+2)!)^2(2h+1)!(2h)!}\geq 1 \]
for all $h\geq 2$. But $a_2=\frac{10!3!2!}{(6!)^2}\geq 1$, and, for all $h\geq 3$, we have
\[ a_h-a_{h-1}=S_h(T_h-1) \]
where
\[ S_h=\frac{(4h-2)!(h-1)!h!(h+1)!(h+2)!}{((2h)!)^2(2h-1)!(2h-2)!}\geq 0 \]
for all $h\geq 2$, and
\[ T_h=\frac{(4h+1)(4h-1)(h+2)(h+3)}{(2h+1)^2(2h+1)(2h+2)}. \]
An easy computation gives that $T_h\geq 1$ if and only if $56h^3+91h^2+h-4\geq 0$, which, in particular, is true for all $h\geq 2$.
Thus, for all $h\geq 2$, $a_h\geq a_2\geq 1$.

If $k$ is odd, then $k=2h+3$, and (\ref{v_h}) is verified if and only if
\[ a'_h:=\frac{(4h+4)!h!(h+1)!(h+3)!(h+4)!}{(2h+3)!(2h+2)!(2h)!(2h+4)!}\geq 1 \]
for all $h\geq2$. Again $a'_2=\frac{11\cdot10\cdot9}{7}\geq 1$, and
\[ a'_h-a'_{h-1}=S'_h(T'_h-1) \]
where
\[ S'_h=\frac{(4h)!(h-1)!h!(h+2)!(h+3)!}{(2h+1)!(2h)!(2h-2)(2h+2)!}\geq 0 \]
for all $h\geq 2$, and
\[ T'_h=\frac{(4h+3)(4h+1)(h+3)(h+4)}{2(2h+3)^2(2h-1)(h+2)} \]
so that $T'_h\geq 1$ if and only if $56h^3+215h^2+207h+72\geq 0$, which, in particular, is true for all $h\geq 2$, and we conclude as before.
\end{proof}

\end{document}